\pdfoutput=1
\RequirePackage{ifpdf}
\ifpdf 
\documentclass[pdftex]{sigma}
\else
\documentclass{sigma}
\fi

 \def\im{\textnormal {Im}}

 \def\E{{\mathcal E}}
 \def\r{{\mathbb R}}
 \def\wr{\textnormal{Wr}}
 \def\c{{\mathbb C}}
 \def\z{{\mathbb Z}}
 \def\n{{\mathbb N}}
 \def\d{{\textnormal d}}
 
 \def\k{\textbf{\textsl{k}}}
 \def\j{\textbf{\textsl{j}}}

\numberwithin{equation}{section}

\newtheorem{Theorem}{Theorem}[section]
\newtheorem*{Theorem*}{Theorem}
\newtheorem{Corollary}[Theorem]{Corollary}
\newtheorem{Lemma}[Theorem]{Lemma}
\newtheorem{Proposition}[Theorem]{Proposition}
\newtheorem{Conjecture}[Theorem]{Conjecture}

\theoremstyle{definition}

\newtheorem{Remark}[Theorem]{Remark}

\begin{document}

\allowdisplaybreaks

\newcommand{\arXivNumber}{2606.02405}

\renewcommand{\PaperNumber}{111}

\FirstPageHeading

\ShortArticleName{Mehler Formula for Wronskians of Hermite Polynomials}

\ArticleName{Mehler Formula for Wronskians\\ of Hermite Polynomials}

\Author{Alexey KUZNETSOV and Minjian YUAN}

\AuthorNameForHeading{A.~Kuznetsov and M.~Yuan}

\Address{Department of Mathematics and Statistics, York University,\\
4700 Keele Street, Toronto, ON, M3J 1P3, Canada}
\Email{\mail{akuznets@yorku.ca}, \mail{yuanm@yorku.ca}}

\ArticleDates{Received June 02, 2026, in final form November 02, 2026; Published online November 11, 2026}

\Abstract{We prove that the bilinear generating function for Wronskians of Hermite polynomials can be expressed as the classical Mehler kernel multiplied by a polynomial, thereby extending the result of Pupasov-Maksimov [\textit{Ann.\ Physics} \textbf{363} (2015), 122--135, arXiv:1502.01778] for exceptional Hermite polynomials. We establish several properties of the polynomials appearing in this extended version of the Mehler formula and present four conjectures about them.}


\Classification{26C05; 33C45}

\section{Introduction and main results}

 Let $H_n(x)$ denote the Hermite polynomials.
It is known (see \cite[formula~(18.18.28)]{NIST}) that for all $|z|<1$ and $x,y\in \c$,
\[
E(z,x,y):=\sum_{n\geqslant 0} z^n \frac{H_n(x) H_n(y)}{2^n n!}=
\frac{1}{\sqrt{1-z^2}} \exp\left( \frac{2xzy-z^2\bigl(x^2+y^2\bigr)}{(1-z^2)} \right).
\]
The expression on the right-hand side is called the Mehler kernel, and it plays a fundamental role in many applications. For example, it is used to construct the kernel of the fractional Fourier transform \cite{Namias1980}, the propagator for the harmonic oscillator in quantum mechanics \cite{Pupasov_2015}, and the transition probability density of the Ornstein--Uhlenbeck process \cite{Bakry2014}.

Before stating our results, we introduce some notation. We denote by $\n$ the set of positive integers and by $\z_{\geqslant 0}$ the set of nonnegative integers. The Wronskian of smooth functions $\{f_j(x)\}_{1\leqslant j \leqslant m}$ is defined as
\[
\wr[f_1,\dots,f_m]:=\det\bigl[\partial_x^{i-1} f_j(x)\bigr]_{1\leqslant i,j \leqslant m}.
\]
For an $m$-tuple $\j=(j_1,j_2,\dots,j_m) \in \z_{\geqslant 0}^m$, we define
\[
H_{\j}(x):=\wr[H_{j_1}(x), \dots,H_{j_m}(x)].
\]
We set $H_{\j}(x) \equiv 1$ if ${\mathbf j}=\varnothing$ (the empty tuple).
In what follows, $\k=(k_1,k_2,\dots,k_l)\in \n^l$ will always denote a strictly increasing sequence of positive integers, and $l$ will denote the length of the sequence $\k$. For $n\in \z_{\geqslant 0}$, we define $H_{\k,n}(x):=H_{\j}(x)$, where $\j=(k_1,k_2,\dots,k_l,n)$.

Wronskians of Hermite polynomials are important objects that arise in several contexts. They play a fundamental role in the classification of rational monodromy-free Schr\"odinger operators with rational potentials having quadratic growth at infinity \cite{Oblomkov1999}. They appear as rational solutions of the fourth Painlev\'e equation \cite{Gomez2021}. Zeros of Wronskians of Hermite polynomials were studied in \cite{Felder2012}. It is known \cite{Adler1994} that $H_{\k}(x)$ has no real zeros if and only if $\k$ is a {\it Krein--Adler sequence}, that is, a strictly increasing sequence $\k=(k_1,k_2,\dots, k_l) \in \n^l$ such that $l$ is even and $k_{2i}=k_{2i-1}+1$ for all $1\leqslant i \leqslant l/2$.

We now define our main object of interest:
\begin{equation}\label{def:Ek}
E_{\k}(z,x,y):=\sum_{n \in \z_{\geqslant 0} \setminus \k} z^n \frac{H_{\k,n}(x) H_{\k,n}(y)}{2^{n+l} n! \prod_{j=1}^l (n-k_j)}.
\end{equation}
As we will show later in Lemma \ref{lemma1}, the series in \eqref{def:Ek} converges for $|z|<1$, $x,y\in \c$ and defines an analytic function of three variables in that domain. Our main result is the following.
\begin{Theorem}\label{thm_main}
For any strictly increasing sequence $\k \in \n^l$, there exist polynomials $Q^{\k}_n(x,y)$, $0\leqslant n \leqslant k_l+1$ such that
\begin{equation}\label{eqn_main}
E_{\k}(z,x,y)=E(z,x,y) \sum_{n=0}^{k_l+1} z^n Q^{\k}_n(x,y),
\end{equation}
for all $|z|<1$, $x,y\in \c$.
\end{Theorem}

Theorem~\ref{thm_main} was established by A.M.~Pupasov-Maksimov~\cite{Pupasov_2015} in the important case where~$\k$ is a Krein--Adler sequence. To explain the significance of this result and the role that Krein--Adler sequences play, let us consider the functions
\begin{equation}\label{def_fkn}
f_{\k,n}(x):={\rm e}^{-x^2/2}
\frac{H_{\k,n}(x)}{H_{\k}(x)}.
\end{equation}
It is known that for every strictly increasing sequence $\k \in \n^l$ these functions
are solutions of the second-order linear differential equation
\begin{equation}\label{f_kn_ODE}
{\mathcal L} f(x):=-f''(x)+U_{\k}(x) f(x)=(2n+1-2l) f(x),
\end{equation}
where $U_{\k}(x)$ is a rational function defined by
\begin{equation*}
U_{\k}(x):= x^2- 2\partial_x^2 \ln(H_{\k}(x))=
 x^2 - 2\frac{H_{\k}''(x)}{H_{\k}(x)}+
 2\frac{H_{\k}'(x)^2}{H_{\k}(x)^2}.
\end{equation*}
This result follows from \cite[Proposition 5.2]{Gomez-Ullate_2014} and it is stated in this exact form in the proof of \cite[Proposition 5.5]{Gomez-Ullate_2014}. The differential operator in \eqref{f_kn_ODE} and the solutions of equation \eqref{f_kn_ODE} can be obtained by a sequence of Darboux transformations, see \cite[Theorem 3]{Oblomkov1999} and \cite{Gomez-Ullate_2014}.

As mentioned above, when $\k$ is a Krein--Adler sequence, the polynomial $H_{\k}(x)$ has no real zeros (see \cite[Theorem~4.2]{Gomez-Ullate_2014}) and the polynomials $H_{\k,n}(x)$ are called {\it exceptional Hermite polynomials}. In this case, the functions $f_{\k,n}(x)$ belong to the Schwartz class (the class of smooth rapidly decreasing functions on $\r$) and satisfy the orthogonality condition
\begin{equation*}
\int_{\r} f_{\k,n}(x) f_{\k,m}(x) \d x= \delta_{n,m} \sqrt{\pi} 2^{n+l} n! \prod_{j=1}^l (n-k_j), \qquad n,m \in {\mathbb Z}_{\geqslant 0} \setminus \k.
\end{equation*}
Moreover, the functions $\{f_{\k,n}(x)\}_{n \in {\mathbb Z}_{\geqslant 0} \setminus \k}$ form a complete orthogonal basis of $L_2(\r, \d x)$. These results can be found in \cite{Duran_2014,Gomez-Ullate_2014,Gomez2016}. Thus, when $\k$ is a Krein--Adler sequence, the functions $\{f_{\k,n}(x)\}_{n \in {\mathbb Z}_{\geqslant 0} \setminus \k}$ form a complete eigenbasis for the operator ${\mathcal L}$ in $L_2(\r, \d x)$. This allows one to write down the spectral expansion of the corresponding heat kernel \cite{Avramidi2015,Vassilevich2003} (the integral kernel of the operator $\exp(-t {\mathcal L})$) in the form
\begin{align}\label{L_heat_kernel}
p(t,x,y)=\sum_{n \in {\mathbb Z}_{\geqslant 0} \setminus \k}
{\rm e}^{-(2n+1-2l) t} \frac{f_{\k,n}(x) f_{\k,n}(y)}{\lVert f_{\k,n} \rVert^2}=\frac{1}{\sqrt{\pi}} {\rm e}^{(2l-1)t-(x^2+y^2)/2}
 \frac{E_{\k}\bigl({\rm e}^{-2t},x,y\bigr)}{H_{\k}(x) H_{\k}(y)}.
\end{align}
Thus, the significance of Theorem~\ref{thm_main} in the Krein--Adler case is that it gives an explicit expression for the heat kernel of the Schr\"odinger operator ${\mathcal L}$.

\begin{Remark} Pupasov-Maksimov \cite{Pupasov_2015} states his results using the probabilists' version of Hermite polynomials, defined by ${\textnormal{He}}_n(x):=2^{-n/2} H_n\bigl(x/\sqrt{2}\bigr)$. Denoting the polynomials appearing in \cite{Pupasov_2015} by $\widetilde Q_n^{\k}(x,y)$, the relation between these two families of polynomials is
\[
\widetilde Q_n^{\k}(x,y)=C Q_n^{\k}\bigl(x/\sqrt{2}, y/\sqrt{2} \bigr),
\]
where $C$ is a normalization constant that depends only on $\k$.
\end{Remark}

In the following three propositions, we state several properties of the polynomials $Q_n^{\k}(x,y)$.

\begin{Proposition}\label{proposition1}
\quad
\begin{itemize}\itemsep=0pt
\item[$(i)$]
For $n=0,1,\dots,k_l+1$, the polynomials $Q_n^{\k}(x,y)$ satisfy
\[
Q^{\k}_n(x,y)=Q_n^{\k}(y,x)=Q^{\k}_n(-x,-y)=(-1)^{\kappa+n} Q^{\k}_n(-x,y),
\] where
$\kappa:=k_1+k_2+\dots+k_l-l(l+1)/2$.
\item[$(ii)$]
We have
\begin{equation}\label{eqn_Qk0}
Q^{\k}_0(x,y)=(-2)^{l} \left[ \prod_{j=1}^l k_j \right] H_{\k-1}(x) H_{\k-1}(y),
\end{equation}
where $\k-1:=(k_1-1,\dots,k_l-1)$.
\item[$(iii)$]
For $n=1,2,\dots,k_l+1$,
the polynomials $Q^{\k}_n(x,y)$ can be computed recursively via
\begin{equation}\label{Q_m_recursion}
Q^{\k}_n(x,y)=
{\mathbf 1}_{\{n \notin \k\}}
\frac{H_{\k,n}(x) H_{\k,n}(y)}{2^{n+l} n! \prod_{j=1}^l (n-k_j)}-
\sum_{m=1}^{n} \frac{H_{m}(x) H_{m}(y)}{2^{m} m!}
Q^{\k}_{n-m}(x,y).
\end{equation}
\end{itemize}
\end{Proposition}

The parity conditions and formula~\eqref{Q_m_recursion} were established in~\cite{Pupasov_2015} for Krein--Adler sequences~$\k$.

\begin{Proposition}\label{proposition2}
Let $\k \in \n^l$ be a Krein--Adler sequence. Define
\begin{align*}
&\Omega(x,y):=\frac{1}{y-x} \left( \frac{H'_{\k}(y)}{H_{\k}(y)}
- \frac{H'_{\k}(x)}{H_{\k}(x)} \right),
\\
&\omega(x):=\lim_{y\to x} \Omega(x,y)=
 \frac{H_{\k}''(x)}{H_{\k}(x)}-\frac{H_{\k}'(x)^2}{H_{\k}(x)^2}.
\end{align*}
For every $x,y\in \c$, the following identities hold:
\begin{align}
\label{sum1}
&\sum_{n=0}^{k_l+1} Q^{\k}_n(x,y)= H_{\k}(x) H_{\k}(y),\\
\label{sum2}
&\sum_{n=0}^{k_l+1} n Q_n^{\k}(x,y)=- H_{\k}(x)
H_{\k}(y) ( \Omega(x,y)-l ),\\
\label{sum3}
&\sum_{n=0}^{k_l+1} n^2 Q_n^{\k}(x,y)= H_{\k}(x)
H_{\k}(y) \left((\Omega(x,y)-l)^2+\frac{\omega(x)+\omega(y)-2 \Omega(x,y)}{(y-x)^2} \right).
\end{align}
\end{Proposition}

Formula \eqref{sum1} for general $x$ and $y$ and formula \eqref{sum2} for the special case $x=y$ were first established in \cite{Pupasov_2015}.

\begin{Proposition}\label{proposition3}
Let $k\in \n$ and $\k=(k)$, so that $l=1$. Then
\begin{gather}
Q_n^{\k}(x,y)={\mathbf 1}_{\{n\geqslant 1\}} 2^{n-1} (k)_{n-1} H_{k+1-n}(x) H_{k+1-n}(y) \nonumber\\
\hphantom{Q_n^{\k}(x,y)=}{} - 2^{n+1} (k)_{n+1}
H_{k-1-n}(x) H_{k-1-n}(y),\label{Q_n_l=1}
\end{gather}
where $(r)_n:=r(r-1)\cdots(r-n+1)$ denotes the falling factorial.
\end{Proposition}

We also state several conjectures, which we have verified numerically for many cases of strictly increasing sequences $\k \in \n^l$.

\begin{Conjecture}\label{Conjecture1} Identities \eqref{sum1}--\eqref{sum3} hold for any strictly increasing sequence $\k \in \n^l$, not only for Krein--Adler sequences.
\end{Conjecture}

\begin{Conjecture}\label{Conjecture2}
For $l\geqslant 2$ and any strictly increasing sequence $\k \in \n^l$, the following identity holds:
\begin{equation*}
Q^{\k}_{k_l+1} (x,y)=2^{k_l+l-1}
(k_l)! \left[ \prod_{j=1}^{l-1} (k_l-k_j) \right] H_{\tilde \k}(x) H_{\tilde \k}(y),
\end{equation*}
where $\tilde \k=(k_1,k_2,\dots,k_{l-1})$.
\end{Conjecture}

\begin{Conjecture}\label{Conjecture3} When $l=2$ and $\k$ is a Krein--Adler sequence $\k=(k,k+1)$, then
\begin{gather*}
Q^{\k}_{k+1}(x,y) =
2^{k+1} k! k \bigl( H_{k+1}(x) H_{k+1}(y)
+ 2 (k+1) [ H_{k+1}(x) H_{k-1}(y)+H_{k+1}(y) H_{k-1}(x) ]\\
\hphantom{Q^{\k}_{k+1}(x,y) =2^{k+1} k! k \bigl(}{} + 4 k (k+1) H_{k-1}(x) H_{k-1}(y) \bigr).
\end{gather*}
\end{Conjecture}

\begin{Conjecture}\label{Conjecture4}
The polynomials $Q^{\k}_n(x,y)$ have integer coefficients. A stronger conjecture is that the scaled polynomials $\Delta^{-2} Q^{\k}_n(x/2,y/2)$ have integer coefficients, where we defined
\[
 \Delta:=\prod_{1\leqslant i < j \leqslant l} (k_i-k_j).
 \]
\end{Conjecture}

Foata \cite{Foata1978} gave a combinatorial proof of the classical Mehler formula. If true, Conjecture~\ref{Conjecture4} may indicate that Theorem~\ref{thm_main} could be proved by combinatorial methods and may suggest a~potential combinatorial interpretation of the scaled polynomials $Q_n^{\k}(x,y)$.

We tested Conjectures~\ref{Conjecture1}--\ref{Conjecture4} numerically for numerous strictly increasing sequences $\k$ satisfying $l\leqslant 6$ and $k_l\leqslant 12$. For 129 of these sequences, we computed the coefficients of the scaled polynomials $\Delta^{-2} Q^{\k}_n(x/2,y/2)$ for all $0\leqslant n \leqslant k_l+1$. In each of these cases, the coefficients were found to be integers. The corresponding data are available for download from Alexey Kuznetsov's \href{https://kuznetsovmath.ca/}{webpage}.

We present the proofs of all our results in the next section.
Our proof of Theorem~\ref{thm_main} follows the same general strategy as the proof of Pupasov-Maksimov \cite{Pupasov_2015} in the case of Krein--Adler sequences. However, the Krein--Adler condition, which implies the orthogonality and completeness of $f_{\k,n}$ in $L_2(\r, \d x)$, played an essential role in the proof in \cite{Pupasov_2015}. For the general case, we therefore had to modify the argument substantially
and introduce several new ideas.

\section{Proofs}\label{section_proofs}

We define by $\psi_n(x):={\rm e}^{-x^2/2} H_n(x)$ the non-normalized Hermite functions. If $n<0$, we set $H_n(x)\equiv 0$ and $\psi_n(x) \equiv 0$. For $\j=(j_1,j_2,\dots,j_m) \in \z_{\geqslant 0}^m$, we define $\psi_{\j}(x):=\wr[\psi_{j_1}(x), \dots,\allowbreak \psi_{j_m}(x)]$ and
\[
|\,\j\,|:=j_1+\dots+j_m.
\]
For a strictly increasing sequence $\k \in \n^l$ and $n \in {\mathbb Z}_{\geqslant 0}$, we define $\psi_{\k,n}(x):=\psi_{\j}(x)$, where $\j=(k_1,k_2,\dots,k_l,n)$. By the following well-known property of Wronskian determinants
\begin{equation}\label{multiplication_identity}
\wr[f(x) g_1(x),f(x) g_2(x) \dots,f(x) g_{m}(x)]=f(x)^m \wr[ g_1(x), g_2(x), \dots, g_{m}(x)],
\end{equation}
we have
\begin{equation}\label{psi_k_H_k}
\psi_{\k}(x)={\rm e}^{-l x^2/2} H_{\k}(x), \qquad \psi_{\k,n}(x)={\rm e}^{-(l+1) x^2/2} H_{\k,n}(x).
\end{equation}
In what follows, we work with the bilinear generating functions of $\psi_n$ and $\psi_{\k,n}$, defined by
\begin{equation}\label{def_mathcal_E}
{\mathcal E}(z,x,y):=\sum_{n\geqslant 0} z^n \frac{\psi_n(x) \psi_n(y)}{2^n n!}=
\frac{1}{\sqrt{1-z^2}} \exp\left( \frac{4xyz-\bigl(1+z^2\bigr)\bigl(x^2+y^2\bigr)}{2\bigl(1-z^2\bigr)} \right),
\end{equation}
and
\begin{equation}\label{def_mathcal_Ek}
{\mathcal E}_{\k}(z,x,y):=\sum_{n \in \z_{\geqslant 0} \setminus \k} z^n \frac{\psi_{\k,n}(x) \psi_{\k,n}(y)}{2^{n+l} n! \prod_{j=1}^l (n-k_j)}.
\end{equation}
From \eqref{psi_k_H_k}, we have
\begin{equation}\label{mathcal_E_E}
{\mathcal E}(z,x,y)={\rm e}^{-(x^2+y^2)/2} E(z,x,y), \qquad
{\mathcal E}_{\k}(z,x,y)={\rm e}^{-(l+1)(x^2+y^2)/2} E_{\k}(z,x,y).
\end{equation}

Before proving Theorem~\ref{thm_main}, we need to establish several auxiliary results.

\begin{Lemma}\label{lemma1}\quad
\begin{itemize}\itemsep=0pt
\item[$(i)$]
There exists $C=C(\k)>0$ such that for all $n\geqslant 1$ and $x\in \c$
\begin{equation}\label{psi_upper_bound}
|\psi_{\k,n}(x)|\leqslant C n^l \bigl| {\rm e}^{-l x^2/2}\bigr| \bigl(1+|x|^{|\k|}\bigr) \sum_{j=0}^l
|\psi_{n-j}(x)|.
\end{equation}
\item[$(ii)$]
 The function $\E_{\k}(z,x,y)$
is an analytic function of three variables in the domain $|z|<1$, $x,y\in \c$.
\item[$(iii)$] The function $\E_{\k}(z,x,y)$ satisfies the following symmetry conditions:
\begin{gather}\label{E_k_symmetries}
\E_{\k}(z,x,y)=\E_{\k}(z,y,x)=\E_{\k}(z,-x,-y)=(-1)^{\kappa} \E_{\k}(-z,-x,y), \\
 |z|<1, \quad x,y\in \c,\nonumber
\end{gather}
where $\kappa=|\,\k\,|-l(l+1)/2$.
\end{itemize}
\end{Lemma}
\begin{proof}
To prove (i), we start with the formula
\[
\psi_{\k,n}(x)={\rm e}^{-(l+1)x^2/2} \wr[H_{k_1}(x),\dots,H_{k_l}(x),H_n(x)],
\]
and expand the Wronskian determinant along the last column, which contains the derivatives of $H_n(x)$:
\begin{equation}\label{psi_k_n_sum_derivatives}
\psi_{\k,n}(x)={\rm e}^{-(l+1)x^2/2} \sum_{j=0}^l p_{j}(x) \partial_x^j H_n(x),
\end{equation}
where
\[
p_{j}(x)=(-1)^{j+l} \det
\bigl[ \partial_x^m H_{k_i}(x) \bigr]_{\substack{1\leqslant i \leqslant l \\ 0\leqslant m \leqslant l \\ m\neq j}}.
\]
It is clear from the above determinant expression that $p_{j}$ is a polynomial of degree not exceeding~$|\, \k \, |$. This, combined with the formula
\[
\partial_x^j H_m(x)=2^j m(m-1)\cdots(m-j+1) H_{m-j}(x),
\]
implies the upper bound~\eqref{psi_upper_bound}.

The proof of item (ii) requires the following asymptotic result (see \cite[Theorem~8.22.7]{Szego}): as $n\to +\infty$
\[
\psi_n(x) = \frac{2^n}{\sqrt{\pi}} \Gamma((n+1)/2) \bigl[ \cos\bigl(x \sqrt{2n+1} - n\pi/2\bigr)+O\bigl(n^{-1/2} \exp\bigl( |\im(x)|\sqrt{2n+1}\bigr)\bigr)\bigr],
\]
uniformly in $x$ on compact subsets of $\c$. We also need the fact that
\[
\frac{\Gamma((n+1)/2)^2}{n!}=2^{-n} \sqrt{\frac{2\pi}{n}} (1+o(1)), \qquad n\to +\infty,
\]
which follows from the Legendre duplication formula for the gamma function.
Combining the above two facts with \eqref{psi_upper_bound}, we conclude that for each compact subset $A \subset \c$ there exists a~constant $C_1=C_1(A,\k)>0$ such that the absolute value of each term in the infinite series~\eqref{def_mathcal_Ek} is bounded above by
\[
C_1 |z|^n n^{l-1/2} \exp\bigl(\sqrt{2n+1} (|\im(x)|+|\im(y)|)\bigr),
\]
for $x,y \in A$. This implies that the series \eqref{def_mathcal_Ek} converges uniformly on compact subsets of $\{(z,x,y) \in \c^3 \mid |z|<1\}$
and defines an analytic function in this domain.

According to \cite[Lemmas~2.1 and~3.6]{Bonneux2018}, the polynomial $H_{\k,n}(x)$ has degree $\kappa+n$ and satisfies the parity condition $H_{\k,n}(-x)=(-1)^{\kappa+n} H_{\k,n}(x)$.
The symmetry conditions in item (iii) follow immediately from this result and formulas~\eqref{psi_k_H_k}, \eqref{def_mathcal_E} and~\eqref{def_mathcal_Ek}.
\end{proof}

\begin{Lemma}\label{lemma2}
Assume that $U(x)$, $g_1(x)$ and $g_2(x)$ are smooth functions on $(-\infty, c)$ and that
\begin{equation}\label{f_i_equation}
 -g_i''(x)+U(x) g_i(x)=\lambda_i g_i(x), \qquad x \in (-\infty, c), \quad i\in \{1,2\}.
\end{equation}
 Assume further that $\lambda_1\neq \lambda_2$ and $\wr[g_1(x),g_2(x)] \to 0$ as $x\to -\infty$. Then, for any $x\in (-\infty, c)$, we have
\[
\int_{-\infty}^x g_1(y)g_2(y) \d y = \frac{1}{ \lambda_1-\lambda_2} \wr[g_1(x),g_2(x)].
\]
\end{Lemma}
\begin{proof} The proof is based on the identity
\[
\frac{\d}{\d x}\wr[g_1(x),g_2(x)] = (\lambda_1-\lambda_2) g_1(x)g_2(x),
\]
which follows from \eqref{f_i_equation}.
\end{proof}

Let ${\mathcal L}_{\k,x}$ be the first-order differential operator acting on the $x$-variable, defined by
\[
{\mathcal L}_{\k,x}h(x):=\wr[ \psi_{\k}(x), h(x)]=\psi_{\k}(x) h'(x)
-\psi_{\k}'(x) h(x),
\]
where $h$ is a smooth function on $\r$.
For a strictly increasing sequence $k=(k_1,\dots,k_l)\in \n^l$, we denote by $\k \setminus k_j$ the sequence of length $l-1$ obtained from $\k$ by removing the term $k_j$. We also define
\[
\tilde \k:=\k \setminus k_l=(k_1,\dots,k_{l-1}).
\] The Wronskian identity
\begin{equation}\label{Wronskian_identity}
\wr[g_1,\dots,g_{j}] \wr[g_1,\dots,g_{j},h,f]
=\wr [\wr[g_1,\dots,g_{j},h],\wr[g_1,\dots,g_{j},f] ]
\end{equation}
 implies the following key result:
\begin{equation}\label{eqn_psi_kn_L_k}
\psi_{\tilde \k}(x) \psi_{\k,n}(x)= {\mathcal L}_{\k,x} \psi_{\tilde \k,n}(x).
\end{equation}

\begin{Lemma}\label{lemma3} Let $\k \in \n^l$ be a strictly increasing sequence. There exists $c=c(\k) \in \r$ such that for $|z|<1$, $x\in \r$ and $y \in (-\infty,c]$, we have
\begin{equation}\label{eqn_E_recursion}
\E_{\k}(z,x,y)=- \frac{\psi_{\tilde \k}(y)}{\psi_{\tilde \k}(x)}
{\mathcal L}_{\k,x} \int_{-\infty}^y
\E_{\tilde \k}(z,x,u) \frac{\psi_{\k}(u)}{\psi_{\tilde \k}(u)^2} \d u.
\end{equation}
\end{Lemma}
\begin{proof}
We take $c \in \r$ such that the polynomial $H_{\tilde \k}(x)$ has no zeros on $(-\infty,c]$. This implies that $\psi_{\tilde \k}(u)=\exp\bigl(-(l-1)u^2/2\bigr) H_{\tilde \k}(u)$ is also non-zero on $(-\infty,c]$. We recall that $f_{\k,n}(x)$ is defined in \eqref{def_fkn} and that these functions satisfy equation \eqref{f_kn_ODE}.

To prove formula \eqref{eqn_E_recursion}, we first apply Lemma \ref{lemma2}:
\begin{align}
\int_{-\infty}^y \E_{\tilde \k}(z,x,u) \frac{\psi_{\k}(u)}{\psi_{\tilde \k}(u)^2} \d u
& =
\sum_{n \in \z_{\geqslant 0} \setminus \tilde \k} z^n
\frac{\psi_{\tilde \k,n}(x)}{2^{n+l-1} n! \prod _{j=1}^{l-1} (n-k_j)} \int_{-\infty}^y f_{\tilde \k,k_l}(u) f_{\tilde \k,n}(u) \d u\nonumber\\
& =\sum_{n \in \z_{\geqslant 0} \setminus \k} z^n
\frac{ \psi_{\tilde \k,n}(x)}{2^{n+l-1} n! \prod_{j=1}^{l-1} (n-k_j)}
\frac{\wr[f_{\tilde\k,k_l}(y),f_{\tilde \k,n}(y)]}{2k_l-2n}
\nonumber\\ & \quad {}+
\frac{z^{k_l} \psi_{\tilde \k,k_l}(x)}{2^{k_l+l-1} k_l! \prod_{j=1}^{l-1} (k_l-k_j)} \int_{-\infty}^y f_{\tilde \k,k_l}(u)^2 \d u.\label{lemma3_proof1}
\end{align}
From \eqref{eqn_psi_kn_L_k}, we find
\[
\wr[f_{\tilde\k,k_l}(y),f_{\tilde \k,n}(y)]=\frac{1}{\psi_{\tilde \k}(y)^2} \wr[\psi_{\k}(y),\psi_{\tilde \k,n}(y)]=
\frac{\psi_{\k,n}(y)}{\psi_{\tilde \k}(y)},
\]
and therefore
\[
\int_{-\infty}^y \E_{\tilde \k}(z,x,u) \frac{\psi_{\k}(u)}{\psi_{\tilde \k}(u)^2} \d u=
-\frac{1}{\psi_{\tilde \k}(y)}\sum_{n \in \z_{\geqslant 0} \setminus \k} z^n
\frac{\psi_{\tilde \k,n}(x) \psi_{\k,n}(y)}{2^{n+l} n! \prod_{j=1}^{l} (n-k_j)}+z^{k_l} \psi_{\k}(x) G(y),
\]
for some function $G(y)$.

Finally, applying identity \eqref{eqn_psi_kn_L_k} once again and noting that ${\mathcal L}_{\k,x} \psi_{\k}(x)\equiv 0$, we obtain
\begin{align}\label{lemma3_proof2}
- \frac{\psi_{\tilde \k}(y)}{\psi_{\tilde \k}(x)}
{\mathcal L}_{\k,x} \int_{-\infty}^y
\E_{\tilde \k}(z,x,u) \frac{\psi_{\k}(u)}{\psi_{\tilde \k}(u)^2} \d u&=
\frac{1}{\psi_{\tilde \k}(x)}\sum_{n \in \z_{\geqslant 0} \setminus \k} z^n
\frac{{\mathcal L}_{\k,x} \psi_{\tilde \k,n}(x) \psi_{\k,n}(y)}{2^{n+l} n! \prod_{j=1}^{l} (n-k_j)}
\\ \nonumber &=
\sum_{n \in \z_{\geqslant 0} \setminus \k} z^n
\frac{\psi_{\k,n}(x) \psi_{\k,n}(y)}{2^{n+l} n! \prod_{j=1}^{l} (n-k_j)}=\E_{\k}(z,x,y).
\end{align}

In the above derivation of formula~\eqref{eqn_E_recursion}, we interchanged summation and differentiation/\allowbreak integration. Let us first justify the interchange of summation and integration in the first step of~\eqref{lemma3_proof1}. Consider the function $f_{\tilde \k,k_l}(u) f_{\tilde \k,n}(u)$, which we rewrite in the form
\[
f_{\tilde \k,k_l}(u) f_{\tilde \k,n}(u)=\left[ {\rm e}^{-u^2/2} \frac{H_{\k}(u)}{H_{\tilde \k}(u)^2} \right] \bigl[ {\rm e}^{(l-1)u^2/2} \psi_{\tilde \k,n}(u)\bigr].
\]
The above factorization follows from \eqref{def_fkn} and \eqref{psi_k_H_k}.
The function ${\rm e}^{-u^2/4} H_{\k}(u)/H_{\tilde \k}(u)^2$ is bounded by some constant $A=A(\k)$ on $(-\infty, c]$ (since $H_{\k}(u)$ and $H_{\tilde \k}(u)$ are polynomials and $H_{\tilde \k}(u)$ is non-zero on $(-\infty, c]$). Thus,
$g(u):={\rm e}^{-u^2/2} H_{\k}(u)/H_{\tilde \k}(u)^2$ satisfies $|g(u)|<A {\rm e}^{-u^2/4}$ for $u \in (-\infty, c]$.
From \eqref{psi_upper_bound} we conclude that there exists a constant $B=B\bigl(\tilde \k\bigr)$ such that
\[
\bigl|{\rm e}^{(l-1)u^2/2} \psi_{\tilde \k,n}(u)\bigr|< B n^{l-1} \bigl(1+|u|^{M}\bigr) \sum_{j=0}^{l-1}
|\psi_{n-j}(u)|, \qquad u \in \r,
\]
where $M=\bigl|\tilde \k\bigr|$.
Applying the Cauchy--Schwarz inequality and using the fact that $\lVert \psi_n \rVert^2=\sqrt{\pi}2^n n!$, we obtain
\begin{align*}
 \int_{-\infty}^y | f_{\tilde \k,k_l}(u) f_{\tilde \k,n}(u) | \d u
 &\leqslant
 A B n^{l-1} \sum_{j=0}^{l-1} \int_{-\infty}^c \bigl(1+|u|^{M}\bigr) {\rm e}^{-u^2/4}
|\psi_{n-j}(u)| \d u
\\ &\leqslant A B n^{l-1} \left[\int_{-\infty}^c \bigl(1+|u|^{M}\bigr)^2 {\rm e}^{-u^2/2}\d u\right]^{1/2}
 \sum_{j=0}^{l-1} \lVert \psi_{n-j}\rVert
\\ &\leqslant C l n^{l-1} \sqrt{2^n n!},
\end{align*}
for some $C=C(\k)>0$.
Using the above estimate and the same argument as in the proof of Lemma \ref{lemma1}, we conclude that the series
\[
\sum_{n \in \z_{\geqslant 0} \setminus \k} r^n
\frac{ |\psi_{\tilde \k,n}(x)|}{2^{n+l-1} n! \prod_{j=1}^{l-1} (n-k_j)} \int_{-\infty}^y | f_{\tilde \k,k_l}(u) f_{\tilde \k,n}(u) | \d u
 \]
 converges for all $r\in (0,1)$, $x\in \c$ and $y \in (-\infty, c]$. Thus, we can apply Fubini's theorem and interchange summation and integration in \eqref{lemma3_proof1}.

To justify the interchange of summation and differentiation in the first step of \eqref{lemma3_proof2}, we use the same argument as in the proof of Lemma \ref{lemma1} and conclude that the series
\[
\sum_{n \in \z_{\geqslant 0} \setminus \k} z^n
\frac{\psi_{\tilde \k,n}(x) \psi_{\k,n}(y)}{2^{n+l} n! \prod_{j=1}^{l} (n-k_j)}
\]
converges uniformly on compact subsets of $\bigl\{(z,x,y) \in \c^3 \mid |z|<1\bigr\}$. Hence, it defines an analytic function on this set, and the infinite series can be differentiated term-by-term.
\end{proof}

For $n\in \n$, $|z|<1$, $x\in \c$ and $y\in \r$, we define
\begin{equation}\label{def_In}
I_n(z,x,y):=\int_{-\infty}^y \E(z,x,u) \psi_n(u) \d u.
\end{equation}

\begin{Lemma}\label{lemma4}
The function $I_n(z,x,y)$ extends to an analytic function in the domain
\[
\bigl\{(z,x,y) \in \c^3 \mid |z|<1\bigr\}
\]
and it satisfies
\begin{equation}\label{eqn_I_n_main}
I_n(z,x,y)=
 z^n \psi_n(x) \frac{\sqrt{\pi}}{2} ( 1+
{\textnormal{erf}}(w) )-{\rm e}^{-x^2/2-w^2} \sum_{i=1}^n \binom{n}{i} \zeta^{i} z^{n-i}
H_{i-1}(w) H_{n-i}(x),
\end{equation}
where $\zeta:=\sqrt{1-z^2}$ and $w:=(y-xz)/\zeta$.
\end{Lemma}
\begin{proof}
We first observe that $\E(z,x,y)$ defined in \eqref{def_mathcal_E} can be written in the form
\begin{equation*}
\E(z,x,y)=\zeta^{-1} {\rm e}^{y^2/2-x^2/2-(y-xz)^2/\zeta^2}.
\end{equation*}
Assume that $z\in (-1,1)$ and $x,y\in \r$.
Using the above expression for $\E(z,x,y)$, we rewrite~\eqref{def_In} in the equivalent form
\[
I_n(z,x,y)=\zeta^{-1} {\rm e}^{-x^2/2} \int_{-\infty}^y {\rm e}^{-(u-xz)^2/\zeta^2} H_n(u) \d u.
\]
We change the variable of integration to $v=(u-xz)/\zeta$ and obtain
\begin{equation*}
I_n(z,x,y)={\rm e}^{-x^2/2} \int_{-\infty}^w {\rm e}^{-v^2} H_n(\zeta v + x z) \d v.
\end{equation*}
Noting that $\zeta^2+z^2=1$, we apply the addition formula for Hermite polynomials (see \cite[formula~(18.18.11)]{NIST})
\begin{equation*}
H_n(\zeta v+xz)=\sum_{i=0}^n \binom{n}{i} \zeta^i z^{n-i}
H_i(v) H_{n-i}(x).
\end{equation*}
The desired result \eqref{eqn_I_n_main} follows from the above two equations and the following formulas:
\begin{equation}\label{two_integrals}
\int_{-\infty}^w {\rm e}^{-v^2} \d v= \frac{\sqrt{\pi}}{2} ( 1+
{\textnormal{erf}}(w) ), \qquad \int_{-\infty}^w {\rm e}^{-v^2} H_i(v) \d v=-{\rm e}^{-w^2} H_{i-1}(w),
\end{equation}
which can be found in \cite[Section~7.2]{NIST} and \cite[equation~(7.373)]{Jeffrey2007}.

Now that we have established~\eqref{eqn_I_n_main} for $z \in (-1,1)$, $x,y\in \r$, we note that the right-hand side in~\eqref{eqn_I_n_main} is an analytic function in $\bigl\{(z,x,y) \in \c^3 \mid |z|<1\bigr\}$. Thus, $I_n(z,x,y)$ can be extended to this domain by analytic continuation.
\end{proof}

We introduce the $l$-th order linear differential operator
${\mathcal M}_{\k,x}$, acting on the $x$-variable, by
\begin{equation}\label{def_M_kx}
{\mathcal M}_{\k,x} h(x):=\wr[ \psi_{k_1},\dots,\psi_{k_l},h](x),
\end{equation}
where $h$ is an entire function. The following result was the starting point for the derivation of the Mehler formula for exceptional Hermite polynomials in \cite{Pupasov_2015}; see also \cite{Pupasov_2007}, where this result was first obtained for Krein--Adler sequences $\k$.

\begin{Lemma}\label{lemma5}
For $|z|<1$, $x,y\in \c$, we have
\begin{equation}\label{eqn_PM}
\E_{\k}(z,x,y)=
{\mathcal M}_{\k,x} \sum_{j=1}^l (-1)^{j+l-1} \psi_{\k \setminus k_j}(y) I_{k_j}(z,x,y).
\end{equation}
\end{Lemma}
\begin{proof}
The proof proceeds by induction on $l$. The case $l=1$ and $\k=(k)$ is covered by Lemma \ref{lemma3}, followed by analytic continuation in $x$ and $y$.
 Assume now that $l\geqslant 2$ and that for $\tilde \k=(k_1,k_2,\dots,k_{l-1})$ we have
\[
\E_{\tilde \k}(z,x,y)=
{\mathcal M}_{\tilde \k,x} \sum_{j=1}^{l-1} (-1)^{j+l-2} \psi_{\tilde \k \setminus k_j}(y) I_{k_j}(z,x,y),
\]
with the above identity valid for all $z \in (-1,1)$ and $x,y\in \r$.
We substitute the above expression for $\E_{\tilde \k}(z,x,y)$ into the right-hand side of \eqref{eqn_E_recursion} and conclude that there exists $c \in \r$ such that for $z\in (-1,1)$, $x\in \r$ and $y \in (-\infty,c]$, we have
\begin{equation}\label{lemma5_proof1}
\E_{\k}(z,x,y)=
 \frac{\psi_{\tilde \k}(y)}{\psi_{\tilde \k}(x)}
{\mathcal L}_{\k,x}{\mathcal M}_{\tilde \k,x} \sum_{j=1}^{l-1} (-1)^{j+l-1} \int_{-\infty}^y
 \psi_{\tilde \k \setminus k_j}(u) I_{k_j}(z,x,u) \frac{\psi_{\k}(u)}{\psi_{\tilde \k}(u)^2} \d u.
\end{equation}
Using the Wronskian identity \eqref{Wronskian_identity}, we check that
\[
\frac{\d}{\d u} \frac{\psi_{\k \setminus k_j}(u)}{\psi_{\tilde \k}(u)}=\psi_{\tilde \k \setminus k_j}(u) \frac{\psi_{\k}(u)}{\psi_{\tilde \k}(u)^2}.
\]
Note that $\psi_{\k \setminus k_j}(u)/\psi_{\tilde \k}(u)$ is a rational function of $u$ (this follows from \eqref{multiplication_identity}). Due to our assumption $z\in (-1,1)$, we have $\zeta>0$, thus from \eqref{eqn_I_n_main} we see that the function $u\mapsto I_{k_j}(z,x,u)$ decays exponentially fast as $u\to -\infty$ (this fact also follows from the integral definition \eqref{def_In}). Thus,
\[
\lim\limits_{u\to -\infty} I_{k_j}(z,x,u)\frac{\psi_{\k \setminus k_j}(u)}{\psi_{\tilde \k}(u)}=0.
\]
Using the above two results and integrating by parts, we obtain
\begin{gather}
\int_{-\infty}^y
 \psi_{\tilde \k \setminus k_j}(u) I_{k_j}(z,x,u) \frac{\psi_{\k}(u)}{\psi_{\tilde \k}(u)^2} \d u
 \nonumber\\
 \qquad =
 I_{k_j}(z,x,y) \frac{\psi_{\k \setminus k_j}(y)}{\psi_{\tilde \k}(y)}-\int_{-\infty}^y \E(z,x,u) \psi_{k_j}(u)
 \frac{\psi_{\k \setminus k_j}(u)}{\psi_{\tilde \k}(u)} \d u.\label{lemma5_proof2}
\end{gather}

Next, we use the identity
\[
0=\det \begin{bmatrix}
\psi_{k_1}(u) & \psi_{k_2}(u) & \cdots & \psi_{k_{l}}(u)\\
\psi_{k_1}(u) & \psi_{k_2}(u) & \cdots & \psi_{k_{l}}(u)\\
\psi_{k_1}'(u) & \psi_{k_2}'(u) & \cdots & \psi_{k_{l}}'(u)\\
\vdots & \vdots & \ddots & \vdots\\
\psi_{k_1}^{(l-2)}(u) & \psi_{k_2}^{(l-2)}(u) & \cdots & \psi_{k_{l}}^{(l-2)}(u)
\end{bmatrix}=\sum_{j=1}^l (-1)^{j-1} \psi_{k_j}(u) \psi_{\k \setminus k_j}(u),
\]
where we have expanded the above determinant along the first row.
Therefore,
\[
\sum_{j=1}^{l-1} (-1)^{j-1} \psi_{k_j}(u) \psi_{\k \setminus k_j}(u) = (-1)^{l} \psi_{k_l}(u) \psi_{\tilde \k}(u).
\]
Combining the above result with \eqref{lemma5_proof1} and \eqref{lemma5_proof2},
we arrive at
\[
\E_{\k}(z,x,y)=
 \frac{1}{\psi_{\tilde \k}(x)}
{\mathcal L}_{\k,x}{\mathcal M}_{\tilde \k,x} \sum_{j=1}^{l} (-1)^{j+l-1}
 \psi_{\k \setminus k_j}(y) I_{k_j}(z,x,y).
\]
From the Wronskian identity \eqref{Wronskian_identity}, it follows that
\[
{\mathcal L}_{\k,x}{\mathcal M}_{\tilde \k,x} h(x) = \psi_{\tilde \k}(x) {\mathcal M}_{\k,x}h(x),
\]
for any entire function $h$. Thus, we have proved that for $z\in (-1,1)$, $x\in \r$ and $y\in (-\infty, c]$, we have
\[
\E_{\k}(z,x,y)=
 {\mathcal M}_{\k,x} \sum_{j=1}^{l} (-1)^{j+l-1}
 \psi_{\k \setminus k_j}(y) I_{k_j}(z,x,y).
\]
By analytic continuation, the above identity holds everywhere in the domain $\bigl\{(z,x,y)\in \c^3 \mid |z|<1\bigr\}$.
This completes the induction step.
\end{proof}

\begin{Corollary}\label{Corollary1}
For $l\geqslant 2$, $|z|<1$, $x\in \c$ and $y\in \r$, we have
\begin{equation}\label{eqn_PM2}
\E_{\k}(z,x,y)=
{\mathcal M}_{\k,x} \int_{-\infty}^y \E(z,x,u) {\rm e}^{-u^2/2} P_{\k}(u,y) \d u,
\end{equation}
where $P_{\k}(u,y)$ is a polynomial in $u$ that satisfies $\partial_u^j P_{\k}(u,y) |_{u=y}=0$ for all $j=0,1,\dots,l-2$.
\end{Corollary}
\begin{proof}
Formula \eqref{eqn_PM2} is equivalent to \eqref{eqn_PM} with $P_{\k}(u,y)$ defined by
\begin{align*}
{\rm e}^{-u^2/2} P_{\k}(u,y)&=\sum_{j=1}^l (-1)^{j+l-1} \psi_{k_j}(u) \psi_{\k \setminus k_j}(y)\\ &=(-1)^l \det \begin{bmatrix}
\psi_{k_1}(u) & \psi_{k_2}(u) & \cdots & \psi_{k_{l}}(u)\\
\psi_{k_1}(y) & \psi_{k_2}(y) & \cdots & \psi_{k_{l}}(y)\\
\psi_{k_1}'(y) & \psi_{k_2}'(y) & \cdots & \psi_{k_{l}}'(y)\\
\vdots & \vdots & \ddots & \vdots\\
\psi_{k_1}^{(l-2)}(y) & \psi_{k_2}^{(l-2)}(y) & \cdots & \psi_{k_{l}}^{(l-2)}(y)
\end{bmatrix}.
\end{align*}
The above equation can be written in the form
\[
P_{\k}(u,y)={\rm e}^{-(l-1)y^2/2} (-1)^l \det\begin{bmatrix}
H_{k_1}(u) & H_{k_2}(u) & \cdots & H_{k_{l}}(u)\\
H_{k_1}(y) & H_{k_2}(y) & \cdots & H_{k_{l}}(y)\\
H_{k_1}'(y) & H_{k_2}'(y) & \cdots & H_{k_{l}}'(y)\\
\vdots & \vdots & \ddots & \vdots\\
H_{k_1}^{(l-2)}(y) & H_{k_2}^{(l-2)}(y) & \cdots & H_{k_{l}}^{(l-2)}(y)
\end{bmatrix},
\]
which immediately implies
$\partial_u^j P_{\k}(u,y) |_{u=y}=0$ for all $j=0,1,\dots,l-2$.
\end{proof}

The following lemma and corollary play a key role in the final step of the proof of Theorem~\ref{thm_main}. They establish that, for real $x>y$, the ratio $\E_{\k}(z,x,y)/\E(z,x,y)$ has a finite limit as $z\to 1^-$, allowing us to show that the rational function $R(z,x,y)$ in \eqref{R_and_P} has removable singularities at $z=\pm 1$ and is therefore a polynomial.

\begin{Lemma}\label{lemma6}
Let $m,n \in {\mathbb Z}_{\geqslant 0}$ and let $x$, $y$ be fixed real numbers satisfying $x>y$. Let $P(u)$ be a~polynomial such that $P^{(i)}(y)=0$ for $i=0,1,\dots,n$. If $m\leqslant n+1$, then
\begin{equation}\label{eqn_integral_limit}
\lim\limits_{z\to 1-} \frac{\partial_x^{m} \int_{-\infty}^y \E(z,x,u) {\rm e}^{-u^2/2} P(u) \d u}{\E(z,x,y)}
=0.
\end{equation}
If $m=n+2$, the above limit exists and is finite.
\end{Lemma}
\begin{proof}
First, we check that formula \eqref{def_mathcal_E} can be rewritten in the equivalent form
\begin{equation}\label{E_factorization}
\E(z,x,y) ={\rm e}^{(x^2+y^2)/2 - 2xy/(1+z)}
\zeta^{-1}
{\rm e}^{-(y-x)^2/\zeta^2 },
\end{equation}
where $\zeta=\zeta(z)=\sqrt{1-z^2}$. For the rest of this proof, we assume that $1/\sqrt{2}<z<1$, so that $0<\zeta<1/\sqrt{2}$.

We claim that, for any $m\in {\mathbb Z}_{\geqslant 0}$ and any fixed $x,y \in \r$ such that $x>y$, we have
\begin{equation}\label{eqn_q_integral}
\zeta^{-2} {\rm e}^{(y-x)^2/\zeta^2} \int_{-\infty}^y
{\rm e}^{-2xu/(1+z)-(u-x)^2/\zeta^2} u^m \d u \to
\frac{{\rm e}^{-x y}y^m}{2(x-y)},
\end{equation}
as $z\to 1-$ (note that $\zeta\to 0+$ as $z\to 1-$).
Indeed, changing the variable of integration $u=y-\zeta^2 v$, we obtain{\samepage
\begin{align*}
&
\zeta^{-2} {\rm e}^{(y-x)^2/\zeta^2} \int_{-\infty}^y
{\rm e}^{-2xu/(1+z)-(u-x)^2/\zeta^2} u^m \d u
\\ &\qquad = {\rm e}^{-2x y/(1+z)} \int_0^{\infty} {\rm e}^{-2(x-y)v -\zeta^2 (v^2-2xv/(1+z))} (y-\zeta^2 v)^m \d v,
\end{align*}
and the limit in \eqref{eqn_q_integral} follows by the dominated convergence theorem.}

Next, let $q(z,u)$ be a polynomial in $u$ of the form
\begin{equation}\label{form_of_q}
q(z,u)=\sum_{i=0}^M a_i(z) u^i,
\end{equation}
where $M$ does not depend on $z$ and the coefficients $a_i(z)$ are continuous in some neighborhood of $z=1$. For $z\in (0,1)$ and $j \in {\mathbb Z}_{\geqslant 0}$, we consider
\begin{equation}\label{def_Q_j}
{\mathcal Q}_{j}(z,x,y):=\zeta^{-1} \int_{-\infty}^y \bigl[\partial_u^j
{\rm e}^{-(u-x)^2/\zeta^2} \bigr] {\rm e}^{-2xu/(1+z)} q(z,u) (u-y)^{n+1} \d u.
\end{equation}
From \eqref{E_factorization} and \eqref{eqn_q_integral}, we conclude that
\begin{equation*}
\frac{{\mathcal Q}_{0}(z,x,y)}{\E(z,x,y)} \to 0
\end{equation*}
as $z \to 1-$. When $0< j \leqslant n+1$, we
integrate by parts $j$ times
and obtain
\begin{align}
&{\mathcal Q}_j(z,x,y)=
\zeta^{-1} \sum_{i=0}^{j-1} (-1)^i \bigl[\partial_u^{j-1-i}
{\rm e}^{-(u-x)^2/\zeta^2} \bigr]_{u=y}\nonumber
\\ &\hphantom{{\mathcal Q}_j(z,x,y)=}{} \times
\bigl[\partial_u^i \bigl(
{\rm e}^{-2xu/(1+z)} q(z,u) (u-y)^{n+1}\bigr)\bigr]_{u=y}
\nonumber\\ &
\hphantom{{\mathcal Q}_j(z,x,y)=}{}
 +(-1)^j \zeta^{-1} \int_{-\infty}^y
{\rm e}^{-(u-x)^2/\zeta^2}
\partial_u^j \bigl(
{\rm e}^{-2xu/(1+z)} q(z,u) (u-y)^{n+1}\bigr) \d u.\label{Q_j_integration_by_parts}
\end{align}
In the finite sum in the above formula, we have $i\leqslant j-1 \leqslant n$, and therefore
\[
\bigl[\partial_u^i
{\rm e}^{-2xu/(1+z)} q(z,u) (u-y)^{n+1}\bigr]_{u=y}=0.
\]
Thus, the finite sum in \eqref{Q_j_integration_by_parts} is equal to zero, and
it follows from \eqref{E_factorization} and \eqref{eqn_q_integral} that, for $0<j\leqslant n+1$,
\begin{equation*}
\frac{{\mathcal Q}_{j}(z,x,y)}{\E(z,x,y)} \to 0,
\end{equation*}
as $z \to 1-$.

When $j=n+2$, the preceding argument shows that the terms with $0\leqslant i \leqslant j-2$ in the finite sum in \eqref{Q_j_integration_by_parts} are all equal to zero, and when $i=j-1=n+1$ we have
\[
\bigl[\partial_u^{n+1} \bigl(
{\rm e}^{-2xu/(1+z)} q(z,u) (u-y)^{n+1}\bigr)\bigr]_{u=y}=
{\rm e}^{-2xy/(1+z)} q(z,y) (n+1)!.
\]
Thus, it follows from
\eqref{E_factorization}, \eqref{eqn_q_integral} and
\eqref{Q_j_integration_by_parts} that, as $z\to 1-$,
\begin{equation*}
\frac{{\mathcal Q}_{n+2}(z,x,y)}{\E(z,x,y)} \to
 (-1)^{n+1}
 {\rm e}^{-(x^2+y^2)/2} q(1,y) (n+1)!.
\end{equation*}

Now we are ready to complete the proof of Lemma \ref{lemma6}. For any fixed $z\in (1/\sqrt{2},1)$, we interchange the integration and differentiation in the numerator in~\eqref{eqn_integral_limit},
use the factorization~\eqref{E_factorization}, and obtain
\begin{equation}\label{numerator_as_sum_of_Q}
\partial_x^{m} \int_{-\infty}^y \E(z,x,u) {\rm e}^{-u^2/2} P(u) \d u
= \sum_{j=0}^m (-1)^j \binom{m}{j} \widetilde {\mathcal Q}_{j}(z,x,y),
\end{equation}
where we have defined
\begin{equation}\label{def_tilde_Q}
\widetilde {\mathcal Q}_{j}(z,x,y):=\zeta^{-1} \int_{-\infty}^y
\bigl[ \partial_u^{j} {\rm e}^{-(u-x)^2/\zeta^2}\bigr] \bigl[ \partial_x^{m-j} {\rm e}^{x^2/2-2xu/(1+z)}\bigr] P(u) \d u.
\end{equation}
When deriving the above formula, we used the fact that
\[
\partial_x^{j} {\rm e}^{-(u-x)^2/\zeta^2}=(-1)^j \partial_u^{j} {\rm e}^{-(u-x)^2/\zeta^2}.
\]
The interchange of integration and differentiation is justified since, for $z\in \bigl(1/\sqrt{2},1\bigr)$, we have $1/\zeta^2>2$; thus, for all $u\in \r$ the integrand in \eqref{def_tilde_Q} is dominated by $C\exp(- u^2)$ for some $C>0$, uniformly in $x$ on compact subsets of $\r$. We note that the conditions $P^{(i)}(y)=0$ for $i=0,1,\dots,n$ imply that $P(u)=q_1(u) (u-y)^{n+1}$ for some polynomial $q_1$, which in turn implies
\[
\bigl[ \partial_x^{m-j} {\rm e}^{x^2/2-2xu/(1+z)} \bigr] P(u)=
{\rm e}^{x^2/2-2xu/(1+z)} q_2(x,u/(1+z),u) (u-y)^{n+1},
\]
where $q_2(x_1,x_2,x_3)$ is some polynomial in three variables. The important conclusion is that $q_3(z,u)=q_2(x,u/(1+z),u)$ is of the form \eqref{form_of_q} (recall that everywhere in this proof $x$ is fixed, so dependence of the coefficients of the polynomial $q_3(z,u)$ on $x$ is not a concern). Thus, $\widetilde {\mathcal Q}_{j}(z,x,y)$ is an integral of the form \eqref{def_Q_j}, and using the results about ${\mathcal Q}_j(z,x,y)$ that we established above, we conclude that the limit
\[
\lim\limits_{z\to 1-} \frac{\widetilde {\mathcal Q}_{j}(z,x,y)}{\E(z,x,y)}
\]
exists and is equal to zero for all $0\leqslant j \leqslant n+1$, while for $j=n+2$ the limit exists and is finite. The desired result now follows from \eqref{numerator_as_sum_of_Q}.
\end{proof}

\begin{Corollary}\label{Corollary2}
Let $x$ and $y$ be fixed real numbers such that $x>y$. Then
\[
\lim_{z\to 1-} \frac{\E_{\k}(z,x,y)}{\E(z,x,y)}
\]
exists and is finite.
\end{Corollary}
\begin{proof}
Using the same reasoning as in the derivation of \eqref{psi_k_n_sum_derivatives}, we conclude that
\begin{equation}\label{M_kx_h}
{\mathcal M}_{\k,x} h(x)={\rm e}^{-l x^2/2} \sum_{j=0}^l q_{j}(x) \partial_x^j h(x),
\end{equation}
for some real polynomials $q_j(x)$, which may also depend on $\k$.
The desired result now follows from Corollary~\ref{Corollary1} and Lemma \ref{lemma6}.
\end{proof}

\begin{proof}[Proof of Theorem~\ref{thm_main}]
We denote by $\r[x]$ the set of polynomials in the $x$-variable with real coefficients, and by $\r_{n}[z,x,y]$ the set of real polynomials in the three variables $z,x,y$ whose degree in $z$ is not greater than $n$.
An expression such as
\[
g(z,x,y) \in H(z,x,y) \sum_{i=1}^n \r[x] h_{i}(z,x,y)
\]
should be interpreted as stating that there exist polynomials $P_i \in \r[x]$ such that
\[
g(z,x,y) = H(z,x,y) \sum_{i=1}^n P_i(x) h_{i}(z,x,y)
\]
for all $z$, $x$, $y$ in the domain of $g$, $H$ and $h_i$. For example, with this notation we can express~\eqref{M_kx_h} in the form
\begin{equation}\label{item_i}
{\mathcal M}_{\k,x} h(x) \in {\rm e}^{-l x^2/2}\sum_{j=0}^l \r[x] \partial_x^j h(x).
\end{equation}

As before, we define $\zeta:=\sqrt{1-z^2}$ and $w:=(y-xz)/\zeta$. We collect some preliminary facts:
\begin{itemize}\itemsep=0pt
\item[(i)] For all $n\in \z_{\geqslant 0}$,
\begin{equation}\label{item_ii}
 H_n(w)=n! \sum_{j=0}^{\lfloor n/2 \rfloor} \frac{(-1)^j (2w)^{n-2j}}{j! (n-2j)!} \in \zeta^{-n} \r_{n}[z,x,y].
\end{equation}
\item[(ii)] For all $n\in {\mathbb N}$,
\begin{equation}\label{item_iii}
\partial_x^n \operatorname{erf} (w)=-(z/\zeta)^n \frac{2}{\sqrt{\pi}} {\rm e}^{-w^2} H_{n-1}(w) \in {\rm e}^{-w^2} \zeta^{1-2n} \r_{2n-1}[z,x,y].
\end{equation}
\item[(iii)] For all $n,m \in \z_{\geqslant 0}$,
\begin{equation}\label{item_iv}
\partial_x^m \bigl[ {\rm e}^{-w^2} H_n(w)\bigr]=(z/\zeta)^m {\rm e}^{-w^2} H_{n+m}(w).
\end{equation}
\end{itemize}
The expression for $H_n(w)$ in item (i) can be found in \cite[formula~(18.5.13)]{NIST}. The results in items~(ii) and~(iii) follow from~\eqref{two_integrals} and the fact that $\partial_x = (-z/\zeta) \partial_w$.

We define
\begin{align*}
&I^{(1)}_n(z,x,y):=
 z^n \psi_n(x) \frac{\sqrt{\pi}}{2} ( 1+
{\textnormal{erf}}(w) ),
\\
&I_n^{(2)}(z,x,y):=-z^n {\rm e}^{-w^2} \sum_{i=1}^n \binom{n}{i} (z/\zeta)^{-i}
H_{i-1}(w) \psi_{n-i}(x).
\end{align*}
By \eqref{eqn_PM}, we have
\[
\E_{\k}(z,x,y) \in {\rm e}^{-(l-1) y^2/2} \sum_{j=1}^l \r[y] \bigl[ {\mathcal M}_{\k,x} I^{(1)}_{k_j}(z,x,y)+{\mathcal M}_{\k,x} I^{(2)}_{k_j}(z,x,y)\bigr].
\]
Using \eqref{item_i}, we obtain
\begin{gather*}
{\mathcal M}_{\k,x} I^{(1)}_{k_j}(z,x,y) \in
 z^{k_j} \frac{\sqrt{\pi}}{2} ( 1+
{\textnormal{erf}}(w) ) {\mathcal M}_{\k,x} \psi_{k_j}(x)\\
\hphantom{{\mathcal M}_{\k,x} I^{(1)}_{k_j}(z,x,y) \in}{} +
 z^{k_j} {\rm e}^{-lx^2/2} \sum_{\substack{n\geqslant 0, m\geqslant 1 \\ m+n\leqslant l}} \r[x] \partial_x^n \psi_{k_j}(x)\partial_x^m {\textnormal{erf}}(w).
\end{gather*}
By definition of ${\mathcal M}_{\k,x}$ (see \eqref{def_M_kx}), we have ${\mathcal M}_{\k,x} \psi_{k_j}(x)\equiv 0$. We use \eqref{item_ii}, \eqref{item_iii}, and the fact that $\partial_x^n \psi_{k_j}(x)\in \exp(-x^2/2) \r[x]$ to conclude that
\begin{align*}
&
{\mathcal M}_{\k,x} I^{(1)}_{k_j}(z,x,y) \in
 z^{k_j} {\rm e}^{-(l+1)x^2/2-w^2} \sum_{m=1}^l \r[x] (z/\zeta)^m H_{m-1}(w)\\
 &\hphantom{{\mathcal M}_{\k,x} I^{(1)}_{k_j}(z,x,y) \in}{} \subset {\rm e}^{-(l+1)x^2/2-w^2} \zeta^{1-2l} \r_{2l-1+k_j}[z,x,y].
\end{align*}
In the last step, we used the fact that for all $m=1,2,\dots,l$
\[
z^{k_j+m} \zeta^{-m} H_{m-1}(w) \in
z^{k_j+m} \zeta^{1-2m} \r_{m-1}[z,x,y] \subset \zeta^{1-2l} \r_{2l-1+k_j}[z,x,y],
\]
which is true since $\zeta^{2l-2m}=(1-z^2)^{l-m}$.
Similarly, using \eqref{item_ii} and \eqref{item_iv}, we conclude that
\begin{align*}
&
{\mathcal M}_{\k,x} I^{(2)}_{k_j}(z,x,y)
\in z^{k_j} {\rm e}^{-(l+1)x^2/2-w^2} \sum_{i=1}^{k_j} \sum_{m=0}^l \r[x] (z/\zeta)^{m-i} H_{m+i-1}(w)
\\ &\hphantom{{\mathcal M}_{\k,x} I^{(2)}_{k_j}(z,x,y)\in}{}
\subset {\rm e}^{-(l+1)x^2/2-w^2} \zeta^{1-2l} \r_{2l-1+k_j}[z,x,y].
\end{align*}
Thus,
\[
\E_{\k}(z,x,y) \in {\rm e}^{-(l+1)x^2/2-(l-1)y^2/2-w^2} \zeta^{1-2l} \r_{2l-1+k_l}[z,x,y].
\]
Using the expression \eqref{E_factorization} and
the above two equations, we arrive at
\[
\frac{\E_{\k}(z,x,y)}{\E(z,x,y)} \in {\rm e}^{-l(x^2+y^2)/2} \bigl(1-z^2\bigr)^{1-l} \r_{2l-1+k_l}[z,x,y].
\]
Thus, we have established that
\begin{equation}\label{R_and_P}
R(z,x,y):={\rm e}^{l(x^2+y^2)/2} \frac{\E_{\k}(z,x,y)}{\E(z,x,y)}=\frac{P(z,x,y)}{(1-z^2)^{l-1}},
\end{equation}
where $P(z,x,y)$ is a real polynomial in three variables, whose degree in the $z$-variable is at most $2l-1+k_l$. We see that $R(z,x,y)$ is a rational function of $z$ with possible poles only at $z=\pm 1$. Now assume that $x$, $y$ are real and $x>y$. Letting $z\to 1-$ and applying Corollary~\ref{Corollary2}, we conclude that $R(z,x,y)$ has a finite limit at $z=1$. This implies that, when $x$, $y$ are real and $x>y$, the polynomial $P(z,x,y)$ is divisible by $(1-z)^{l-1}$. This divisibility property holds for all $x,y\in \c$. Indeed, consider the Taylor coefficients of the numerator $P(z,x,y)$ at $z=1$. These Taylor coefficients are polynomials in $x$, $y$ and the fact that $P(z,x,y)$ is divisible by $(1-z)^{l-1}$ implies that the first $l-1$ coefficients are equal to zero. It is clear that if a polynomial in two variables $x$, $y$ is zero for all real values $x>y$, then it must be equal to zero for all $x,y\in \c$.

By the symmetry conditions \eqref{E_k_symmetries}, the rational function $R$ satisfies
\[
R(z,x,y)=(-1)^{\kappa} R(-z,-x,y),
\]
which implies that it also has a finite limit as $z\to -1$. Thus, the polynomial $P(z,x,y)$ is divisible by \smash{$\bigl(1-z^2\bigr)^{l-1}$}. Since \smash{$\bigl(1-z^2\bigr)^{l-1}$} has degree $2l-2$, and the numerator $P(z,x,y)$ has degree in~$z$ not greater than $2l-1+k_l$, we conclude that the function $R(z,x,y)$ must be a~polynomial in the three variables $z$, $x$, $y$, with degree in~$z$ not greater than $k_l+1$.

Thus, we have established that there exists a polynomial $R(z,x,y)$, with degree in~$z$ not greater than $k_l+1$, such that
\[
\E_{\k}(z,x,y)={\rm e}^{-l(x^2+y^2)/2} \E(z,x,y) R(z,x,y),
\]
which is equivalent to the desired result \eqref{eqn_main} by \eqref{mathcal_E_E}.
\end{proof}

\begin{Remark}As was shown by Haese-Hill, Halln\"as and Veselov \cite{HaeseHill_2016}, even when $\k$ is not a~Krein--Adler sequence, the polynomials $H_{\k,n}(x)$ satisfy an orthogonality relation with respect to a Hermitian form that is not necessarily positive definite. This form is defined by integration along a suitable contour in the complex plane that avoids the zeros of $H_{\k}(x)$. This observation may provide a route to a more direct proof of Theorem~\ref{thm_main}, although we have not found such a~proof.
\end{Remark}
\begin{Remark} One possible way to prove Conjecture~\ref{Conjecture2} is to identify the coefficient of the highest power of $z$ in the polynomial $P(z,x,y)$ in \eqref{R_and_P} (namely, the coefficient of $z^{2l-1+k_l}$). Indeed, after dividing $P(z,x,y)$ by \smash{$\bigl(1-z^2\bigr)^{l-1}$}, this would give \smash{$Q^{\k}_{k_l+1}(x,y)$} -- the coefficient of~$z^{k_l+1}$ in the polynomial function $R(z,x,y)$. Thus, in order to prove Conjecture~\ref{Conjecture2}, one would need to follow the steps of the proof of Theorem~\ref{thm_main} and keep track of the coefficient of the highest power of~$z$ in every term. However, the resulting expressions become complicated very quickly, and we were unable to complete this computation.
\end{Remark}

\begin{proof}[Proof of Proposition~\ref{proposition1}]
We rewrite \eqref{eqn_main} in the form
\begin{equation}\label{generating_functions}
\sum_{n \in \z_{\geqslant 0} \setminus \k} z^n \frac{H_{\k,n}(x) H_{\k,n}(y)}{2^{n+l} n! \prod_{j=1}^l (n-k_j)}= \sum_{n\geqslant 0} z^n \frac{H_n(x) H_n(y)}{2^n n!} \sum_{m=0}^{k_l+1} z^m Q_m^{\k}(x,y).
\end{equation}
Comparing the coefficient of $z^0$, we obtain
\[
Q_0^{\k}(x,y)= \frac{H_{\k,0}(x) H_{\k,0}(y)}{(-2)^{l} \prod_{j=1}^l k_j}.
\]
Since \smash{$H_{\k,0}(x)=(-2)^l \bigl[\prod_{j=1}^l k_j\bigr] H_{\k-1}(x)$}, which follows from \cite[Lemma~3.1]{Ullate_2018}, we obtain formula~\eqref{eqn_Qk0} for $Q_0^{\k}(x,y)$.

Comparing the coefficients of $z^n$ in \eqref{generating_functions} gives the recursion identity \eqref{Q_m_recursion}. The parity conditions for $Q_n^{\k}(x,y)$ follow from~\eqref{eqn_main} and Lemma~\ref{lemma1}\,(iii).
\end{proof}

\begin{proof}[Proof of Proposition~\ref{proposition2}]
The proof is based on the following short-time expansion of the heat kernel. This result can be found in \cite[Section~5.9]{Avramidi2015}, see also \cite{Iliev2005,Vassilevich2003}.
Let $p(t,x,y)$ be the heat kernel of the operator $L=-\partial_x^2+V(x)$, where $V\colon \r \mapsto \r$ is a smooth function. As $t\to 0^+$, we have the asymptotic expansion
\begin{equation}\label{Hadamard_expansion}
p(t,x,y)= \frac{1}{\sqrt{4\pi t}}
\exp\left( -\frac{(y-x)^2}{4t} \right)
\bigl[ 1 + u_1(x,y) t + u_2(x,y) t^2+O\bigl(t^3\bigr) \bigr],
\end{equation}
where
\begin{equation}\label{eqn_u1}
u_1(x,y)=-\int_0^1 V(x+s(y-x)) \d s,
\end{equation}
and
\begin{equation}\label{eqn_u2}
u_2(x,y)=\frac{1}{2} u_1(x,y)^2 - \int_0^1
s(1-s) V''(x+s(y-x)) \d s.
\end{equation}
The heat kernel for the operator ${\mathcal L}^{(1)}=-\partial_x^2+x^2$ is
\[
p^{(1)}(t,x,y)=\frac{ {\rm e}^{-t}}{\sqrt{\pi}} \E\bigl({\rm e}^{-2t},x,y\bigr),
\]
since ${\mathcal L}^{(1)} \psi_n=(2n+1) \psi_n$ and the functions $\psi_n$ form an orthogonal basis for $L_2(\r,\d x)$ with $\lVert \psi_n \rVert^2=\sqrt{\pi} 2^n n!$.
Let $\omega(x):=\partial_x^2 \ln(H_{\k}(x))$. From \eqref{L_heat_kernel}, we find that the heat kernel for the operator
\[
{\mathcal L}^{(2)}=-\partial_x^2+U_{\k}(x)+2l=-\partial_x^2+x^2-2\omega(x)+2l,
\]
 is
\[
p^{(2)}(t,x,y)=\frac{{\rm e}^{-t} \E_{\k}\bigl({\rm e}^{-2t},x,y\bigr)}{\sqrt{\pi} \psi_{\k}(x) \psi_{\k}(y)}.
\]
From \eqref{eqn_main} and \eqref{mathcal_E_E}, we conclude that
\begin{equation}\label{p_ratio1}
\frac{p^{(2)}(t,x,y)}{p^{(1)}(t,x,y)}=\frac{1}{H_{\k}(x) H_{\k}(y)}
\sum_{n=0}^{k_l+1} {\rm e}^{-2nt} Q_n^{\k}(x,y).
\end{equation}
Let $\alpha_1(x,y)$ and $\alpha_2(x,y)$ (respectively, $\beta_1(x,y)$ and $\beta_2(x,y)$) be the coefficients $u_1(x,y)$ and $u_2(x,y)$ defined by equations \eqref{eqn_u1} and \eqref{eqn_u2} with $V(x)=x^2+2l-2\omega(x)$ (respectively, $V(x)=x^2$). For simplicity, we write $\alpha_i=\alpha_i(x,y)$ and similarly for $\beta_i$. According to \eqref{Hadamard_expansion}, we have
\begin{equation}\label{p_ratio2}
\frac{p^{(2)}(t,x,y)}{p^{(1)}(t,x,y)}=
\frac{1+\alpha_1 t+\alpha_2 t^2+O\bigl(t^3\bigr)}{1+\beta_1 t + \beta_2 t^2 + O\bigl(t^3\bigr)}.
\end{equation}

We compute $\beta_i$ using formulas \eqref{eqn_u1} and \eqref{eqn_u2} with $V(x)=x^2$:
\[
\beta_1=-\frac{1}{3} \bigl(x^2+xy+y^2\bigr), \qquad \beta_2=\frac{1}{2} \beta_1^2 - \frac{1}{3}.
\]

Next, we use the fact that $\omega(x)=\partial_x^2 \ln(H_{\k}(x))$ and compute
\begin{equation*}
\Omega(x,y):=\int_0^1 \omega(x+s(y-x)) \d s=\frac{1}{y-x} \left( \frac{H'_{\k}(y)}{H_{\k}(y)}
- \frac{H'_{\k}(x)}{H_{\k}(x)} \right).
\end{equation*}
After integrating by parts twice, we arrive at
\begin{equation*}
\Theta(x,y):=\int_0^1
s(1-s)\omega''(x+s(y-x)) \d s=
\frac{1}{(y-x)^2} \bigl[\omega(x)+\omega(y)-2\Omega(x,y)\bigr].
\end{equation*}
Again, in what follows, we will write $\Omega=\Omega(x,y)$ and $\Theta=\Theta(x,y)$.
Using formulas~\eqref{eqn_u1} and~\eqref{eqn_u2} with $V(x)=x^2+2l-2\omega(x)$, we find
\[
\alpha_1=\beta_1-2l+2 \Omega, \qquad
\alpha_2=\frac{1}{2} \alpha_1^2 - \frac{1}{3} + 2 \Theta.
\]

We check that
\[
\frac{1+\alpha_1 t+\alpha_2 t^2+O\bigl(t^3\bigr)}{1+\beta_1 t + \beta_2 t^2 + O\bigl(t^3\bigr)}=
1+\gamma_1 t + \gamma_2 t^2 + O\bigl(t^3\bigr),
\]
where
$\gamma_1=\alpha_1-\beta_1=2 (\Omega-l)$,
and
$\gamma_2=\alpha_2-\beta_2-\gamma_1 \beta_1 = 2 (\Omega-l)^2 + 2\Theta$.
Combining the above formulas with~\eqref{p_ratio1} and~\eqref{p_ratio2}, we arrive at the following result: for all $x,y\in\r$, as $t\to 0^+$,
\begin{gather*}
\frac{1}{H_{\k}(x) H_{\k}(y)}
\sum_{n=0}^{k_l+1} {\rm e}^{-2nt} Q_n^{\k}(x,y)\\
\qquad =1+2 (\Omega(x,y)-l) t + 2 \bigl[ (\Omega(x,y)-l)^2 + \Theta(x,y)\bigr] t^2 + O\bigl(t^3\bigr).
\end{gather*}
Comparing the first three coefficients in the Taylor expansions at $t=0$ on both sides of the above equation gives us the desired formulas~\eqref{sum1},~\eqref{sum2} and~\eqref{sum3}.
\end{proof}

\begin{Remark} Note that the asymptotic expansion~\eqref{p_ratio2} implies the assertion of Corollary~\ref{Corollary2}, which was used crucially in the proof of Theorem~\ref{thm_main}. Thus, when $\k$ is a Krein--Adler sequence, Corollary~\ref{Corollary2} is a simple consequence of the short-time heat-kernel expansion.
\end{Remark}

\begin{proof}[Proof of Proposition~\ref{proposition3}]
Let $Q_n^k(x,y)$ be the polynomials defined in \eqref{Q_n_l=1}. Since in this case $\k=(k)$, we will write $Q_n^k(x,y)$ instead of $Q_n^{\k}(x,y)$. Similarly, we will write $H_{k,n}(x)=\wr[H_k(x),H_n(x)]$ instead of $H_{\k,n}(x)$.
We will show that $L_n^k(x,y)=R_n^k(x,y)$ for all $k\geqslant 0$, $0\leqslant n \leqslant k+1$ and all $x,y\in \c$, where
\[
L_n^k(x,y):=\sum_{m=0}^{n}
\frac{H_m(x)H_m(y)}{2^m m!}
Q_{n-m}^k(x,y), \qquad
R_n^k(x,y):=
\mathbf 1_{\{n\ne k\}}
\frac{H_{k,n}(x)H_{k,n}(y)}
{2^{n+1}n!(n-k)}.
\]
According to Proposition~\ref{proposition1}, this implies that $Q_n^k(x,y)$ are the polynomials appearing in~\eqref{eqn_main}.

We define
\[
h_m(x,y):=H_m(x)H_m(y).
\]
To simplify notation, we write $h_m=h_m(x,y)$ and similarly for $Q_n^k$, $L_n^k$, and $R_n^k$. From equation~\eqref{Q_n_l=1}, it follows that for $n\geqslant 0$ and $k\geqslant 1$,
\[
Q_n^k={\mathbf 1}_{\{n\geqslant 1\}} 2k Q_{n-1}^{k-1} +
{\mathbf 1}_{\{n=1\}} h_k-2k {\mathbf 1}_{\{n=0\}} h_{k-1}.
\]
Thus, for $n\geqslant 1$ and $k\geqslant 1$,
\begin{align*}
L_n^k& =\sum_{m=0}^{n}
\left[ \frac{h_m}{2^m m!}\right]
Q_{n-m}^k\\
& =2 k \sum_{m=0}^{n-1}
\left[
\frac{h_m}{2^m m!}\right]
Q_{n-1-m}^{k-1}+\left[\frac{h_{n-1}}{2^{n-1} (n-1)!}\right] h_k
-2k \left[\frac{h_{n}}{2^{n} n!}\right] h_{k-1},
\end{align*}
and we obtain the recurrence relation
\begin{equation}\label{L_recursion}
L_n^k=2 k L_{n-1}^{k-1}+\frac{nh_k h_{n-1}-k h_{k-1} h_n}{2^{n-1} n!}.
\end{equation}

Next, using the facts
\[
H_n'(x)=2n H_{n-1}(x)=2x H_n(x)-H_{n+1}(x),
\]
we check that
\begin{align*}
&H_{k,n}(x)=2n H_k(x) H_{n-1}(x)-2k H_{k-1}(x)H_n(x), \\
&H_{k-1,n-1}(x)=H_{k}(x)H_{n-1}(x) -H_{k-1}(x) H_{n}(x).
\end{align*}
Using the above two formulas and a straightforward but tedious calculation, we verify that
\[
H_{k,n}(x)H_{k,n}(y)=4n(n-k)h_k h_{n-1} - 4 k (n-k) h_{k-1} h_n+4 n k H_{k-1,n-1}(x) H_{k-1,n-1}(y).
\]
When $n\geqslant 1$ and $n\neq k$, we divide both sides of the above identity
by $2^{n+1} n! (n-k)$ and obtain
\begin{equation}\label{R_recursion}
R_n^k=2 k R_{n-1}^{k-1}+\frac{nh_k h_{n-1}-k h_{k-1} h_n}{2^{n-1} n!}.
\end{equation}
We check that the above identity also holds for $n=k$, since in this case $R_n^k=R_{n-1}^{k-1}=nh_k h_{n-1}-k h_{k-1} h_n=0$.

We see from formulas \eqref{L_recursion} and \eqref{R_recursion} that $L_n^k$ and $R_n^k$ satisfy the same recurrence relation, which lowers both indices $n$ and $k$ by one. We want to prove the identity $L_n^k=R_n^k$ for $k\geqslant 0$ and $0\leqslant n \leqslant k+1$. After applying the recurrence formulas~\eqref{L_recursion} and~\eqref{R_recursion} $M=\min(n,k)$ times, we arrive at either the case $n=0$, $k\geqslant 0$ or $n=1$, $k=0$. In other words, if the identity $L_n^k=R_n^k$ holds for $n=0$, $k\geqslant 0$ and $n=1$, $k=0$, then it holds for all $k\geqslant 0$ and $n=0,1,\dots,k+1$. The fact that the identity $L_n^k=R_n^k$ holds in these boundary cases is easily verified directly using the formulas
\[
H_{k,0}(x)=-2k H_{k-1}(x), \qquad H_{0,n}(x)=2n H_{n-1}(x),
\]
and the definition \eqref{Q_n_l=1} of $Q_n^k(x,y)$.
\end{proof}

\subsection*{Acknowledgements}
The research was supported by the Natural Sciences and Engineering Research Council of Canada.
We are grateful to anonymous referees for carefully
reading the paper and for providing very helpful comments and suggestions.

\pdfbookmark[1]{References}{ref}
\LastPageEnding


\begin{thebibliography}{99}
\footnotesize\itemsep=0pt

\bibitem{Adler1994}
Adler V.E., A modification of {C}rum's method,
 \href{https://doi.org/10.1007/BF01035458}{\textit{Theoret. and Math. Phys.}}
 \textbf{101} (1994), 1381--1386.

\bibitem{Avramidi2015}
Avramidi I.G., Heat kernel method and its applications,
 \href{https://doi.org/10.1007/978-3-319-26266-6}{Birkh\"auser}, Cham, 2015.

\bibitem{Bakry2014}
Bakry D., Gentil I., Ledoux M., Analysis and geometry of {M}arkov diffusion
 operators, \textit{Grundlehren Math. Wiss.}, Vol.~348,
 \href{https://doi.org/10.1007/978-3-319-00227-9}{Springer}, Cham, 2014.

\bibitem{Bonneux2018}
Bonneux N., Stevens M., Recurrence relations for {W}ronskian {H}ermite
 polynomials, \href{https://doi.org/10.3842/SIGMA.2018.048}{\textit{SIGMA}}
 \textbf{14} (2018), 048, 29~pages,
 \href{https://arxiv.org/abs/1801.07980}{arXiv:1801.07980}.

\bibitem{Duran_2014}
Dur\'an A.J., Exceptional {C}harlier and {H}ermite orthogonal polynomials,
 \href{https://doi.org/10.1016/j.jat.2014.03.004}{\textit{J.~Approx. Theory}}
 \textbf{182} (2014), 29--58,
 \href{https://arxiv.org/abs/1309.1175}{arXiv:1309.1175}.

\bibitem{Felder2012}
Felder G., Hemery A.D., Veselov A.P., Zeros of {W}ronskians of {H}ermite
 polynomials and {Y}oung diagrams,
 \href{https://doi.org/10.1016/j.physd.2012.08.008}{\textit{Phys.~D}}
 \textbf{241} (2012), 2131--2137,
 \href{https://arxiv.org/abs/1005.2695}{arXiv:1005.2695}.

\bibitem{Foata1978}
Foata D., A combinatorial proof of the {M}ehler formula,
 \href{https://doi.org/10.1016/0097-3165(78)90066-3}{\textit{J.~Combin. Theory
 Ser.~A}} \textbf{24} (1978), 367--376.

\bibitem{Gomez-Ullate_2014}
G\'omez-Ullate D., Grandati Y., Milson R., Rational extensions of the quantum
 harmonic oscillator and exceptional {H}ermite polynomials,
 \href{https://doi.org/10.1088/1751-8113/47/1/015203}{\textit{J.~Phys.~A}}
 \textbf{47} (2014), 015203, 27~pages,
 \href{https://arxiv.org/abs/1306.5143}{arXiv:1306.5143}.

\bibitem{Ullate_2018}
G\'omez-Ullate D., Grandati Y., Milson R., Durfee rectangles and
 pseudo-{W}ronskian equivalences for {H}ermite polynomials,
 \href{https://doi.org/10.1111/sapm.12225}{\textit{Stud. Appl. Math.}}
 \textbf{141} (2018), 596--625,
 \href{https://arxiv.org/abs/1612.05514}{arXiv:1612.05514}.

\bibitem{Gomez2021}
G\'omez-Ullate D., Grandati Y., Milson R., Complete classification of rational
 solutions of {$A_{2n}$}-{P}ainlev\'e systems,
 \href{https://doi.org/10.1016/j.aim.2021.107770}{\textit{Adv. Math.}}
 \textbf{385} (2021), 107770, 51~pages,
 \href{https://arxiv.org/abs/2010.00076}{arXiv:2010.00076}.

\bibitem{Gomez2016}
G\'omez-Ullate D., Kasman A., Kuijlaars A.B.J., Milson R., Recurrence relations
 for exceptional {H}ermite polynomials,
 \href{https://doi.org/10.1016/j.jat.2015.12.003}{\textit{J.~Approx. Theory}}
 \textbf{204} (2016), 1--16,
 \href{https://arxiv.org/abs/1506.03651}{arXiv:1506.03651}.

\bibitem{Jeffrey2007}
Gradshteyn I.S., Ryzhik I.M., Table of integrals, series, and products, 7th~ed., Academic Press, Amsterdam, 2007.

\bibitem{HaeseHill_2016}
Haese-Hill W.A., Halln\"as M.A., Veselov A.P., Complex exceptional orthogonal
 polynomials and quasi-invariance,
 \href{https://doi.org/10.1007/s11005-016-0828-8}{\textit{Lett. Math. Phys.}}
 \textbf{106} (2016), 583--606,
 \href{https://arxiv.org/abs/1509.07008}{arXiv:1509.07008}.

\bibitem{Iliev2005}
Iliev P., On the heat kernel and the {K}orteweg--de {V}ries hierarchy,
 \href{https://doi.org/10.5802/aif.2154}{\textit{Ann. Inst. Fourier
 (Grenoble)}} \textbf{55} (2005), 2117--2127,
 \href{https://arxiv.org/abs/math-ph/0504045}{arXiv:math-ph/0504045}.

\bibitem{Namias1980}
Namias V., The fractional order {F}ourier transform and its application to
 quantum mechanics,
 \href{https://doi.org/10.1093/imamat/25.3.241}{\textit{J.~Inst. Math. Appl.}}
 \textbf{25} (1980), 241--265.

\bibitem{Oblomkov1999}
Oblomkov A.A., Monodromy-free {S}chr\"odinger operators with quadratically
 increasing potentials,
 \href{https://doi.org/10.1007/BF02557204}{\textit{Theoret. and Math. Phys.}}
 \textbf{121} (1999), 1574--1584.

\bibitem{NIST}
Olver F.W.J., Olde~Daalhuis A.B., Lozier D.W., Schneider B.I., Boisvert R.F.,
 Clark C.W., Miller B.R., Saunders B.V., Cohl H.S., McClain M.A. (Editors), {NIST}
 {D}igital {L}ibrary of {M}athematical {F}unctions, {R}elease 1.2.6 of
 2026-03-15, \url{https://dlmf.nist.gov/}.

\bibitem{Pupasov_2007}
Pupasov A.M., Samsonov B.F., G\"unther U., Exact propagators for {SUSY}
 partners,
 \href{https://doi.org/10.1088/1751-8113/40/34/013}{\textit{J.~Phys.~A}}
 \textbf{40} (2007), 10557--10587,
 \href{https://arxiv.org/abs/math-ph/0702088}{arXiv:math-ph/0702088}.

\bibitem{Pupasov_2015}
Pupasov-Maksimov A.M., Propagators of isochronous an-harmonic oscillators and
 {M}ehler formula for the exceptional {H}ermite polynomials,
 \href{https://doi.org/10.1016/j.aop.2015.09.021}{\textit{Ann. Physics}}
 \textbf{363} (2015), 122--135,
 \href{https://arxiv.org/abs/1502.01778}{arXiv:1502.01778}.

\bibitem{Szego}
Szeg\H{o} G., Orthogonal polynomials, \textit{Amer. Math. Soc. Colloq. Publ.},
 Vol.~23, 4th ed., \href{https://doi.org/10.1090/coll/023}{American
 Mathematical Society}, Providence, RI, 1975.

\bibitem{Vassilevich2003}
Vassilevich D.V., Heat kernel expansion: {U}ser's manual,
 \href{https://doi.org/10.1016/j.physrep.2003.09.002}{\textit{Phys. Rep.}}
 \textbf{388} (2003), 279--360,
 \href{https://arxiv.org/abs/hep-th/0306138}{arXiv:hep-th/0306138}.

\end{thebibliography}
\end{document}